\documentclass[numbers]{imsart}

\usepackage{amsthm}
\usepackage{amsfonts}
\usepackage{graphicx}
\usepackage{subfig}
\usepackage{latexsym}
\usepackage{url}
\usepackage{bbm}
\usepackage{mathtools}
\mathtoolsset{showonlyrefs}

\startlocaldefs
\theoremstyle{plain}

\newtheorem{prop}{Proposition}[section]
\newtheorem{cor}[prop]{Corollary}
\newtheorem{theo}[prop]{Theorem}
\newtheorem{lem}[prop]{Lemma}
\theoremstyle{remark}

\newtheorem{rem}[prop]{Remark}
\newtheorem{ex}[prop]{Example}

\newcommand{\R}{\mathbb{R}}

\newcommand{\E}{\mathbb{E}}
\newcommand{\N}{\mathbb{N}}

\renewcommand{\P}{\mathbb P}

\newcommand{\al}{\alpha}
\newcommand{\be}{\beta}
\newcommand{\la}{\lambda}

\newcommand{\ga}{\gamma}

\newcommand{\ep}{\varepsilon}

\newcommand{\de}{\delta}
\newcommand{\te}{\theta}

\newcommand{\Te}{\Theta}
\newcommand{\De}{\Delta}

\newcommand{\ia}{\mathcal I}

\newcommand{\ca}{\mathcal C}

\newcommand{\f}{\mathcal F}

\newcommand{\schw}{\stackrel{\raisebox{-1pt}{\textup{\tiny d}}}{\longrightarrow}}

\newcommand{\wte}{\widehat{\theta}}

\DeclareMathOperator*{\argmin}{argmin}

\newcommand{\lan}{\langle}
\newcommand{\ran}{\rangle}

\newcommand{\bee}{\begin{equation}}
\newcommand{\eee}{\end{equation}}
\newcommand{\beea}{\begin{array}}
\newcommand{\eeea}{\end{array}}

\newcommand{\thetal}{\widehat{\theta}_L}

\newcommand{\supptheta}{\mathcal{S}(\theta)}

\newcommand{\btheta}{b_\theta (X_t)}
\newcommand{\bvartheta}{b_\vartheta (X_t)}

\endlocaldefs

\begin{document}

\begin{frontmatter}
\title{On Lasso estimator for the drift function  in diffusion models}
\runtitle{On Lasso estimator for the drift function  in diffusion models}

\begin{aug}
\author[A]{\fnms{Gabriela}~\snm{Cio\l{}ek}\ead[label=e1]{gabriela.ciolek@uni.lu}},
\author[B]{\fnms{Dmytro}~\snm{Marushkevych}\ead[label=e2]{dmytro.m@math.ku.dk}}
\and
\author[A]{\fnms{Mark}~\snm{Podolskij}\ead[label=e3]{mark.podolskij@uni.lu}}
\address[A]{Department of Mathematics, University of Luxembourg \printead[presep={,\ }]{e1,e3}}

\address[B]{Department of Mathematics,
University or Copenhagen \printead[presep={,\ }]{e2}}
\end{aug}

\begin{abstract}
In this paper we study the properties of the Lasso estimator of the drift component in the diffusion setting. More specifically, 
we consider a multivariate parametric diffusion model $X$ observed continuously over the interval $[0,T]$ and investigate drift estimation
under sparsity constraints. We allow the dimensions of the model and the parameter space to be large. We obtain an oracle inequality for the Lasso
estimator and derive an error bound for the $L^2$-distance using concentration inequalities for linear functionals of diffusion processes. The probabilistic part 
is based upon elements of empirical processes theory and, in particular, on the chaining method.
\end{abstract}



\end{frontmatter}

\section{Introduction} \label{sec1}
\setcounter{equation}{0}
\renewcommand{\theequation}{\thesection.\arabic{equation}}

During past decades numerous studies have been devoted to statistical inference for diffusion models. Researchers typically focus on 
parametric and nonparametric estimation of the drift and volatility components under various sampling schemes. We refer to excellent monographs
\cite{JP12,KS97,K04} for an extensive study in high and low frequency settings.  

While most papers consider a fixed dimensional framework, high dimensional statistical problems received much less attention in the field of diffusion processes. The only notable exceptions are recent articles \cite{CMP20,DS22,F19,GM19,PI14} that investigate parametric inference for high dimensional Ornstein-Uhlenbeck models or related processes with \textit{linear} drift.  Their approach is based upon the Lasso method and Dantzig selection, which have been popularised in the context of classical discrete linear models; see e.g.   \cite{BV,Can07} among many other references. Due to a more complex probabilistic structure of diffusion models, statistical analysis of Lasso and Dantzig estimators are far from trivial even in the Ornstein-Uhlenbeck setting. In particular, concentration inequalities for functionals of diffusion processes, which build a backbone for statistical applications, turn out to pose serious challenges. 

In this paper we focus on high dimensional parametric estimation of the drift function in diffusion models. 
We consider a $d$-dimensional ergodic diffusion process given as a strong solution of the stochastic differential equation 
 \bee \label{DP}
dX_t= - b_\theta(X_t) dt + dW_t, \qquad t \geq 0,
\eee
\noindent where $\theta \in \Theta\subseteq  \mathbb{R}^p$ is the model parameter, $\Te$ is an open set, $b_\theta: \R^d\to\R^d$, $W=(W_t)_{t\geq 0}$ is a $d$-dimensional Brownian motion and the initial value $X_0$ is a random vector independent of $W$ (here we suppress the dependence of $X$ on $\te$). The underlying observations are $(X_t)_{t\in [0,T]}$ and we assume that $p$, $d$ and $T$ are large. Such high dimensional models find numerous applications in physics, economics and biology among other sciences. For instance, high dimensional diffusions are studied in the context  of particle systems within \textit{mean field theory}. Probabilistic results in this topic can be found in the classical works  \cite{M66,M67,S91}  while applications of particle systems in sciences are studied in \cite{BCC11,CZ16,FTC09} among many others. More recently, statistical inference for large particle systems and their mean field limits has been investigated in \cite{BPP22,DH21,GL20,K90,SKPP21}. However, the authors only consider a finite dimensional parameter space $\Theta$. 

This work aims at investigating the statistical behaviour of the Lasso estimator of the unknown parameter $\te_0$ under the sparsity constraint:
\bee\label{sparsity}
\| \theta_0\|_0:=\#\{1\leq i \leq p : \theta_{0,i}\neq 0  \} = s_0. 
\eee
Our approach is based on the penalised maximum likelihood estimation and we derive an oracle inequality for the resulting estimator as well as statements about the error bound for the $L^2$-distance in the non-asymptotic framework. The main theoretical results heavily use empirical processes theory and the method of generic chaining, and the upper bound is explicitly expressed in terms of $p$, $d$, $T$ and certain entropy numbers. To the best of our knowledge this is the first theoretical study of the Lasso method in the setting of general drift functions  and high dimensional parameter space $\Te$.  We emphasise that the mathematical methodology is much more involved than the approach in \cite{CMP20,GM19} as the latter use very specific properties of Ornstein-Uhlenbeck models; in particular, the linearity of the drift and the quadratic form of the likelihood function are absolutely crucial in this context.  Related works include \cite{DI12} that studies Lasso estimation in the finite dimensional setting and  \cite{S17} that investigates general theory for maximum likelihood estimators with quadratic penalisations. We remark however that their methods can not be applied in our context.     

The rest of the paper is structured as follows. Section \ref{sec2} introduces basic assumptions, main examples and the classical statistical theory in the finite dimensional setting. Section \ref{sec3} is devoted to the derivation of the oracle inequality for the Lasso estimator. In Section \ref{sec4} we introduce the necessary probabilistic background and obtain some concentration inequalities.    We derive the error bound for the Lasso estimator in Section \ref{sec5} and consider results of some numerical experiments in Section \ref{secNE}. Finally, Section \ref{sec6} contains proofs of the main results.

\subsection*{Notation} 
All vectors are understood as column vectors. For a matrix $A$ we denote by $A^{\star}$
the transpose of $A$.
For $x\in \R^d$ we write $\|x\|_{q}$, $q\in [1,\infty]$, 
to denote the $l_q$-norm of $x$; we set $\| \cdot \|:=\| \cdot \|_2$. The operator norm of a matrix $A\in \R^{d\times d}$ is denoted by $\|A\|_{\text{op}}$; for 
$A,B \in \R^{d\times d}$ we write $A\geq B$ if $B-A$ is positive semidefinite. 
For a set $\mathcal S \subset \{1,\ldots, d\}$ and $x\in \R^d$
we use the notiation
\[
(x|_{\mathcal S})_j: = x_j 1_{\{j \in  \mathcal S \}} \qquad j=1,\ldots, d.
\]
For a function $f: \Te \times \R^d \to \R^l$ we denote by $\dot f$ (resp. $\nabla f$)  the derivative in $\te\in \Te$ (resp. in $x \in  \R^d$). We say that a function $f:  \R^d \to \R^l$ has \textit{polynomial growth} when $\|f(x)\|\leq C (1+\|x\|^q)$ for some $q,C>0$. We write $\|f\|_{\text{Lip}}$ to denote the Lipschitz constant if $f$
is Lipschitz continuous. For a continuous martingale $(M_t)_{t\geq 0}$ we denote by 
$(\langle M \rangle_t)_{t\geq 0}$ its quadratic variation process.

\section{Overview and examples} \label{sec2}
\setcounter{equation}{0}
\renewcommand{\theequation}{\thesection.\arabic{equation}}

In this section we introduce the main assumptions on the model   \eqref{DP} and briefly review classical results about maximum likelihood estimation in finite dimension. We consider the following conditions:
\begin{align}
&\text{The function } b_\theta:\R^d \to \R^d \text{ is locally Lipschitz,}
\label{bcond1}\\[1.5 ex]
&(b_\theta(x)-b_\theta(y))^{\star}(x-y) \geq M \| x-y \|^2  \label{bcond2}
\end{align}
for a constant $M>0$, $x,y \in \R^d$ and $\te \in \Te$.
These conditions imply that equation \eqref{DP} has a unique strong solution, which is a homogeneous continuous Markov process and exhibits an invariant distribution (cf.  \cite[Theorem 12.1]{RogWil}). 
\noindent 
Finally, we assume that initial value $X_0$ follows the invariant law such that the process $X$ is strictly stationary.

We suppose that  the complete path $(X_t)_{t\in[0,T]}$ is observed and we are interested in estimating the unknown parameter $\theta_0\in \Te$. 
Let us briefly recall the classical maximum likelihood theory when $d$ and $p$ are fixed, and $T\to \infty$. We denote by $\P_{\theta}^T$ (resp. $\P_{0}^T$)  the law of the process \eqref{DP} with parameter $\theta$ (resp. drift zero) restricted to $\f_{T}$.  The (scaled) negative log-likelihood function is explicitly computed via Girsanov's theorem as  
\bee \label{loglike}
\mathcal{L}_T(\theta) = \log(\P_{\theta}^T/ \P_{0}^T) = \frac{1}{T} \int_0^T b_\theta^{\star}(X_t)dX_t + \frac{1}{2T}\int_0^T \|b_\theta(X_t) \|^2dt.
\eee 
and we define the maximum likelihood estimator by
\bee\label{MLEdef}
\wte_{\text{M}}=\argmin_{\theta \in \Te} \mathcal{L}_T(\theta). 
\eee 
To obtain the asymptotic normality of the estimator $\wte_{\text{M}}$ the following conditions are imposed: \\ 

\noindent
($\mathcal{A}$): (a) We assume that for any function $f_{\te}:\R^d \to \R$ with polynomial growth and any $q>0$, there exists $C_q>0$ such that
\[
\sup_{\te \in \Te} \E_{\te} \left[\left| \frac{1}{T} \int_0^T \left(f_{\te}(X_t) - \E_{\te}[f_{\te}(X_0)] \right) dt \right|^q \right] \leq C_q T^{-q/2}.
\]
(b) The drift function $b$ is continuously differentiable in $\te$ such that $\dot b_{\theta}(\cdot)$ has polynomial growth.  Furthermore, $\dot b$ is \textit{uniformly continuous}, i.e. for any compact set $K\subset \Te$ it holds that 
\[
\lim_{\de\to 0} \sup_{\te_0 \in K} \sup_{\|\te-\te_0\|<\de} \E_{\te_0} [\|\dot b_{\te}(X_0) - \dot b_{\te_0}(X_0)\|^2] = 0 
\] 
(c) Define the information matrix 
\[
I(\te):= \E_{\te}[\dot b_{\te}(X_0) \dot b_{\te}(X_0)^{\star}].
\]
We assume that $I(\te)$ is positive definite for all $\te \in \Te$. Furthermore, 
for any compact set $K\subset \Te$ it holds that 
\[
\inf_{\te \in K}~ \inf_{e\in \R^p: ~\|e\|=1} e^{\star} I(\te) e>0 
\]
and 
\[
\inf_{\te_0 \in K} \inf_{\|\te-\te_0\|>\de} \E_{\te_0}[\|b_{\te}(X_0) -  b_{\te_0}(X_0)\|^2]>0
\]
for any $\de>0$. \\

\noindent
We emphasise that the last condition in ($\mathcal{A}$)(c) ensures the identifiability of the parameter $\te$. The following result corresponds to \cite[Theorem 2.8]{K04} adapted to the setting of the model  \eqref{DP} with $d>1$. 

\begin{theo}
Consider the model \eqref{DP} and assume that conditions \eqref{bcond2}, ($\mathcal{A}$) hold. 
Then we obtain the central limit theorem
\[
\sqrt{T} (\wte_{\text{M}} - \te_0) \schw \mathcal N_p\left(0, I(\te_0)^{-1}\right).
\]
\end{theo} 

\noindent
To conclude this section we demonstrate some examples of multivariate diffusion models, which are commonly used in the literature.

\begin{ex} \label{ex1} \rm
(i) (Ornstein-Uhlenbeck model) For a matrix $A\in \R^{d\times d}$ consider the diffusion model
\[
dX_t= - AX_t dt + dW_t.
\]
It is well known that this SDE possesses an ergodic strong solutions if all eigenvalues of $A$ have strictly positive real parts.  
High dimensional estimation of Ornstein-Uhlenbeck model has been discussed in  \cite{CMP20,DS22,GM19}. \\ \\
(ii) (General linear models) We may consider extensions of the Ornstein-Uhlenbeck model to more general linear functions.
In this setting we are interested in the model
\[
b_\theta = \phi_0+ \sum_{j=1}^p \theta_j \phi_j,
\]
where $\theta_j>0$ and the functions $\phi_j: \R^d \to\R^d$ satisfy the condition \eqref{bcond2}.
High dimensional inference for linear models have been investigated in \cite{F19,PI14}. \\ \\
(iii) (Langevin equation) Assume that there exists a potential $V_\theta \in C^1(\R^d;\R)$ such that
\[
b_\theta=\nabla V_\theta.
\]
This type of models is frequently used in Monte Carlo simulations. According to  \cite[Theorem 3.5]{Bhat}, under condition \eqref{bcond2} the diffusion process $X$ is ergodic with unique invariant probability measure that is absolutely continuous with respect to the Lebesgue measure and its density is given by 
\bee\label{invmeasure}
\mu_\theta(x)= \left( \int_{\R^d} \exp(-2V_{\theta}(y)) dy \right)^{-1} \exp(-2V_{\theta}(x)), \qquad x\in \R^d. 
\eee
\qed
\end{ex}

\section{Oracle inequality for the Lasso estimator} \label{sec3}
\setcounter{equation}{0}
\renewcommand{\theequation}{\thesection.\arabic{equation}}

In this section we turn our attention to large $p$/large $d$/large $T$ setting. We consider the diffusion model \eqref{DP} and assume that the unknown parameter 
$\theta_0$ satisfies the sparsity constraint \eqref{sparsity}. A standard approach to estimate $\theta$ under the sparsity constraint is the Lasso method, which has been investigated in \cite{CMP20,GM19}  in the framework of an Ornstein-Uhlenbeck model. We define the Lasso estimator as the penalised MLE:
\bee \label{lasso}
\thetal= \argmin_{\theta \in \Theta}\big( \mathcal{L}_T(\theta) + \lambda \| \theta \|_1\big),
\eee
where $\la>0$ is a tuning parameter and the (scaled) negative log-likelihood $ \mathcal{L}_T(\theta)$ has been introduced at  \eqref{loglike}.  
Our first goal is to obtain a basic inequality, which provides a basis for the oracle inequality. For this purpose we introduce some notations. 
For two smooth functions $f,g: \R^d \to \R^d$ we define a random bi-linear form as
\bee
\lan f,g \ran_{T}:= \frac{1}{T} \int_0^T f^{\star}(X_t) g(X_t) dt
\eee
and set $\|f\|_{T}:= \sqrt{\lan f,f \ran_{T}}$. Furthermore, for $\theta$, $\vartheta\in \R^p$ we consider the following random function:
\bee\label{empbias2}
G(\theta, \vartheta) = \frac{1}{T} \int_0^T \big( \btheta - \bvartheta \big)^{\star} dW_t. 
\eee
By definition of the Lasso estimator we obtain our first result.
\begin{lem} \label{BasicIn} (Basic inequality)
For any $\theta \in \R^p$ it holds that 
\bee \label{basic}
\|b_{\thetal} - b_{\theta_0} \|_{T}^2
 \leq \|b_{\theta} - b_{\theta_0} \|_{T}^2  +2 G(\theta, \thetal)   +2 \lambda \big( \| \theta\|_1 -\|  \thetal \|_1 \big).
\eee
\end{lem}

\begin{proof}
The proof is rather standard.  We can write
\bee \label{loglike2}
\mathcal{L}_T(\theta) = \frac{1}{T}  \int_0^T b_\theta^{\star}(X_t) dW_t + \frac{1}{2} \left(  \|b_{\theta} - b_{\theta_0} \|_{T}^2  -  \|b_{\theta_0} \|_{T}^2\right).  
\eee 
By definition of $\thetal$ we also have that $\mathcal{L}_T(\thetal) + \lambda \| \thetal \|_1 \leq \mathcal{L}_T(\theta) + \lambda \| \theta \|_1$
for any $\te\in \Te$. This inequality together with \eqref{loglike2} implies the desired result.
\end{proof}

\noindent In the next step we would like to show an oracle inequality, which holds on sets of high probability. As common 
in the literature we require a good control of the empirical part $G(\theta, \thetal)$ as well as a version of the restricted eigenvalue property.\
For this purpose we introduce the sets $\mathcal{T}$ and $\mathcal{T}'$ via
\bee \label{ScT}
\mathcal{T}:= \left\{\left \| \sup_{\te \in \Te} \frac{1}{T} \int_0^T \dot b_{\te}(X_t)^{\star} dW_t  \right \|_{\infty} \leq \la/2 \right\}
\eee
\bee \label{ScTprime}
\mathcal{T}':= \left\{\inf_{\substack{\theta \in \R^p: \| \theta \|_0=s \\ \vartheta\in \R^p: \theta-\vartheta \in \mathcal{C}(s,3+4/\gamma)}} \frac{\|b_{\theta} - b_{\vartheta} \|_{T}}{\|  \theta - \vartheta\|_2}\geq k \right \}
\eee
Here the set $\mathcal{C}(s,c)$ is given via
\bee \label{Csc0}
\ca(s,c):= \left\{ x \in \R^{d}\setminus\{ 0 \} : ~\|x\|_1 \leq (1+c) \|x_{| \ia_s(x)}\|_1 \right\},
\eee
where $c>0$ and $\ia_s(x)$ is a set of coordinates of $s$ largest elements of $A$. The constant  $k$ in the definition of the set $\mathcal{T}'$
will be chosen later, while the constant $\ga>0$ remains arbitrary. On the set $\mathcal{T}$ we can control the empirical part $G(\theta, \thetal)$ while $\mathcal{T}'$, which is a version of the restricted eigenvalue condition, provides a connection between the empirical norm $\|\cdot\|_T$
and the Euclidean norm $\|\cdot\|_2$. 
The main result of this section is the following 
oracle inequality. 

\begin{theo} \label{mainth}
Assume that $\|\te\|_0=s$. On $\mathcal{T} \cap \mathcal{T}'$ it holds that 
\bee\label{basicineq}
\|b_{\thetal} - b_{\theta_0} \|_{T}^2 \leq (1+\gamma) \|b_{\theta} - b_{\theta_0} \|_{T}^2 +
4 \frac{(\gamma+2)^2s \lambda^2}{\gamma k^2}. 
\eee
\end{theo}

\begin{proof}
 Applying mean value theorem we conclude that there exists a measurable $\te' \in \R^p$ such that 
\begin{align}
|G(\theta, \thetal)| &= \left| \left \lan \frac{1}{T} \int_0^T \dot b_{\te'}(X_t)^{\star} dW_t ,
\te -  \thetal \right \ran \right| \\[1.5 ex]
&\leq \left \| \sup_{\te \in \Te} \frac{1}{T} \int_0^T \dot b_{\te}(X_t)^{\star} dW_t  \right \|_{\infty}
\|\te -  \thetal \|_1.
\end{align}
On the set $\mathcal{T}$ we conclude from the basic inequality  \eqref{basic} that 
\bee
\|b_{\thetal} - b_{\theta_0} \|_{T}^2 + \lambda \| \thetal -\theta  \|_1 \leq 
\|b_{\theta} - b_{\theta_0} \|_{T}^2+2 \lambda ( \| \theta \|_1 - \| \thetal \|_1+\| \thetal -\theta  \|_1).
\eee
\noindent Let $\supptheta=\text{supp }\theta$. Then 
\bee\label{ineqsupp}
\| \theta \|_1 - \| \thetal \|_1+\| \thetal -\theta  \|_1 \leq 2 \| \thetal |_{\supptheta}-\theta \|_1 
\eee
and hence
\bee\label{ineqsupp2}
\|b_{\thetal} - b_{\theta_0} \|_{T}^2 + \lambda \| \thetal -\theta  \|_1 \leq \|b_{\theta} - b_{\theta_0} \|_{T}^2+ 4 \lambda \| \thetal |_{\supptheta}-\theta \|_1. 
\eee
Now if $4 \lambda \| \thetal |_{\supptheta}-\theta \|_1\leq \gamma \|b_{\theta} - b_{\theta_0} \|_{T}^2$ then $\|b_{\thetal} - b_{\theta_0} \|_{T}^2\leq (1+\gamma)\|b_{\theta} - b_{\theta_0} \|_{T}^2$. Next, we consider the case
\bee
4 \lambda \| \thetal |_{\supptheta}-\theta \|_1> \gamma \|b_{\theta} - b_{\theta_0} \|_{T}^2. 
\eee
In this setting, due to \eqref{ineqsupp2} we get $\thetal - \theta \in \mathcal{C}(s,3+4/\gamma)$ and by Cauchy-Schwarz inequality 
\bee\label{ineqsupp3}
\|b_{\thetal} - b_{\theta_0} \|_{T}^2 + \lambda \| \thetal -\theta  \|_1 \leq \|b_{\theta} - b_{\theta_0} \|_{T}^2+ 4 \lambda \sqrt{s} \| \thetal -\theta \|_2. 
\eee
On $\mathcal{T}'$ we obtain via \eqref{ineqsupp3}:
\bee\label{ineqsupp4}
\|b_{\thetal} - b_{\theta_0} \|_{T}^2 \leq \|b_{\theta} - b_{\theta_0} \|_{T}^2+ 4 \lambda \frac{\sqrt{s} }{k} \big( \|b_{\thetal} - b_{\theta_0} \|_{T}+\|b_{\theta} - b_{\theta_0} \|_{T} \big). 
\eee
Using the inequality 
\bee
2xy\leq \frac{x^2}{a}+ay^2
\eee
with $x=\lambda \sqrt{s_0}$, $y=\|b_{\thetal} - b_{\theta_0} \|_{T}$ or 
$y=\|b_{\theta} - b_{\theta_0} \|_{T}$ and $a=\gamma/(\gamma+2)$ we finally obtain that 
\bee
\|b_{\thetal} - b_{\theta_0} \|_{T}^2 \leq (1+\gamma) \|b_{\theta} - b_{\theta_0} \|_{T}^2 +
4 \frac{(\gamma+2)^2 s \lambda^2}{\gamma k^2},
\eee
which completes the proof of Theorem \ref{mainth}.
\end{proof}

\begin{rem} \label{rem3.3} \rm
In linear models discussed in Examples \eqref{ex1}(i) and (ii) the mean value theorem used in the previous proof is trivial, and hence the uniformity in $\te \in \Te$ in the definition of the sets $\mathcal{T}$ and $\mathcal{T}'$ is not required. In this setting it is easier to control the probabilities $\P(\mathcal{T})$ and $\P(\mathcal{T}')$ (cf. \cite{CMP20,GM19}). In the general framework of the model \eqref{DP} we need more advanced techniques from empirical processes theory. \qed
\end{rem}

\noindent
Applying the oracle inequality \eqref{basicineq} to $\te=\te_0$ we immediately conclude that 
\bee \label{estTdis}
\|b_{\thetal} - b_{\theta_0} \|_{T}^2 \leq 
4 \frac{(\gamma+2)^2 s_0 \lambda^2}{\gamma k^2}
\eee
and this statement can be used to obtain bounds on $l_1$ and $l_2$ norms. Indeed, on $\mathcal{T}'$ it holds that 
$k \|  \thetal - \theta_0\|_2 \leq \|b_{\thetal} - b_{\theta_0} \|_{T}$ and we deduce the bound
\bee \label{l2bound}
\|  \thetal - \theta_0\|_2^2 \leq 
4 \frac{(\gamma+2)^2 s_0 \lambda^2}{\gamma k^4}
\eee
On the other hand, since $\thetal - \theta_0 \in \mathcal{C}(s_0,3+4/\gamma)$ we obtain that 
\bee \label{l1bound}
\|  \thetal - \theta_0\|_1 \leq \frac{4(\ga+1)}{\ga}   \sqrt{s_0} \|  \thetal - \theta_0\|_2 \leq \frac{8(\ga+1)(\ga+2) s_0 \la}{\ga^{3/2} k^2}. 
\eee
by applying Cauchy-Schwarz inequality and \eqref{l2bound}.

In the next section we will determine an upper bound for the probabilities $\P(\mathcal{T})$ and $\P(\mathcal{T}')$ with the help of uniform concentration inequalities.

\section{Concentration bounds} \label{sec4}
\setcounter{equation}{0}
\renewcommand{\theequation}{\thesection.\arabic{equation}}
In this section we review some concentration bounds for linear functionals of the diffusion process $X$ defined in \eqref{DP}. We will apply these theoretical  results to control the probabilities of the sets $\mathcal{T}$ and $\mathcal{T}'$.

\subsection{Concentration bounds for linear functionals of the SDE} \label{sec4.1}

We first consider the random set $\mathcal{T}'$ and ignore the uniformity for the moment. From the definition of the norm $\|\cdot\|_T$ it is obvious 
that we require concentration bounds for functionals of the form
\bee \label{linfun}
\frac{1}{T} \int_0^{T} f(X_t) dt
\eee 
for functions $f:\R^d \to \R$. Similarly,  the quadratic variation of the martingale term $\int_0^T \dot b_{\te}(X_t)^{\star} dW_t$ in the definition of $\mathcal{T}$ 
has the form \eqref{linfun}.  In view of Bernstein type inequalities for martingales, which crucially depend on the quadratic variation, it is evident that deviation bounds for linear functionals of $X$ are crucial to obtain approximations of $\P(\mathcal{T})$ and $\P(\mathcal{T}')$. 
We now present an exponential type bound for Lipschitz functions $f:\R^d \to \R$, which will be one of the key tools to assess the probabilities $\P(\mathcal{T})$ and $\P(\mathcal{T}')$.
 
 \begin{theo}[Theorem 3 and Section 3 in \cite{Sa12}] \label{expbound}
Assume that conditions  \eqref{bcond2} are satisfied and let $f: \R^d \to \R$ be a Lipschitz function. Then there exists a constant $C>0$,
independent of $d$,
such that 
\bee \label{concbound}
\E_{\te}\left[ \exp\left( \frac{\mu}{T} \int_0^T (f(X_t) - \E_{\te}[f(X_0)])dt \right) \right] \leq \exp \left(C \mu^2 \|f\|_{\text{\rm Lip}}^2/T \right)
\eee
for all $\mu \in \R$.
\end{theo}

\noindent
Theorem \ref{expbound} is a basic building block for uniform bounds investigated in the next section.
\begin{rem} \label{remExp} \rm
Concentration bounds for functionals of the form  \eqref{linfun} have been investigated in several papers, see e.g. \cite{GGW14,GLWY09,Varvenne} among others. It is important to note that there are no minimal assumptions on $f$ and $X$ under which such concentration bounds hold. Rather there is a certain 
trade-off between assumptions on $f$ and the diffusion model $X$. Here we give a short overview about the theory, mainly following the exposition of \cite{GGW14}. In the setting of model \eqref{DP} concentration inequalities usually take the form
\bee \label{conbound}
\P_{\te}\left( \frac{1}{T} \int_0^{T} \left( f(X_t) - \E_{\te}[f(X_0)] \right)dt >\mu \right) \leq \exp\left(- \frac{T\mu^2}{a_1+a_2\mu} \right)
\eee
where the constants $a_1,a_2>0$ depend on the properties of $f$ and $X$.  Apart from the Lipschitz case presented in Thoerem  \ref{expbound}, which gives sub-Gaussian bounds, concentration bound of type  \eqref{conbound} can be obtained 
under the conditions that (i) $X$ is a symmetric Markov process and the function $f$ is bounded,
(ii) $X$ is a symmetric Markov process that satisfies a (local) Poincar{\'e} inequality and $f$ fulfils certain growth and integrability conditions (see also the recent work \cite{AST22} for new deviation bounds). We remark that in the aforementioned settings the concentration bounds are not sub-Gaussian. Hence, the main results of our paper  (Theorem \ref{controlT} and Corollary \ref{maincor}) are not easily obtained under a different set of conditions on $X$ and $f$. 
 \qed
\end{rem}

\subsection{Generic chaining and control of the probabilities $\P(\mathcal{T})$ and $\P(\mathcal{T}')$} \label{sec4.2}

Here we apply the exponential bound introduced  in Theorem \ref{expbound} to obtain upper bounds for the probabilities $\P(\mathcal{T})$ and $\P(\mathcal{T}')$. We will use the generic chaining method, which has been introduced by Talagrand, cf. \cite{Talagrand2005}. We recall the basic notions of the theory.  Let $\text{d}$ be a distance measure on the parameter set $\Te$
and assume that $\max_{\theta,\vartheta \in \Theta} \text{d}(\theta,\vartheta)$ is finite. 
The diameter of a set $A \subset \Te$ with respect to $\text{d}$ is defined as
\bee
\De(A)=\De(A, \text{d}):= \max_{\theta,\vartheta \in A} \text{d}(\theta,\vartheta)<\infty.
\eee
We call an admissible sequence an increasing sequence $(\mathcal{A}_n)_{n\in \N}$ of partitions of $\Theta$ such that  $\text{card } \mathcal{A}_0=1$ and $\text{card } \mathcal{A}_n\leq 2^{2^n}$ for $n \geq 1$. Furthermore, we introduce the entropy numbers ($\al>0$)
\bee
\gamma_{\alpha}(\Theta,d)=\inf \sup_{\theta \in \Theta} \sum_{n\geq 0} 2^{n/\alpha} \Delta(A_n(\theta)), 
\eee
\noindent where the infimum is taken over all admissible sequences and $A_n(\theta)$ denotes a the unique element of $\mathcal{A}_n$ which contains $\theta$. The generic chaining approach is a tool to obtain \textit{uniform} probability bounds from non-uniform ones. To sketch  ideas we demonstrate a result from \cite{D15}. Consider a one-dimensional family $\{Z_{\theta}:~\te \in \Te\}$ and assume that the inequality
\bee
\P(|Z_{\theta} - Z_{\vartheta}|> u \text{d}(\theta,\vartheta)) \leq 2 \exp(-u^{\al})
\eee
holds for some  $\al>0$. Then, according to \cite[Theorem 3.2]{D15}, there exist constants $a_{\al},b_{\al}>0$ such that the uniform probability bound
\bee
\P\left( \sup_{\te\in \Te} |Z_{\theta} - Z_{\theta_0}|>a_{\al} \gamma_{\alpha}(\Theta,d) + ub_{\al}
\Delta(\Te)\right) \leq \exp(-u^{\al}/\al)
\eee 
holds for any $u\geq 1$. We will use this type of uniform bounds to treat the sets $\mathcal{T}$
and $\mathcal{T}'$.

We introduce two distance measures on $\Te$ that are related to definitions of $\mathcal{T}$
and $\mathcal{T}'$ as well as impose a set of assumptions. \\

\noindent 
($\mathcal{B}$) (a) We denote by 
$\dot b_{\te,j} $ the $j$th column vector of the matrix $\dot b_{\te} $, $j=1,\ldots,p$. The functions $f_{\te,\vartheta}^j(\cdot):= 
\|\dot b_{\te,j}(\cdot) - \dot b_{\vartheta,j}(\cdot)\|^2_2: \R^d \to \R$ and $ (\dot b_\theta \dot b_\theta^{\star})_{ij}(\cdot): \R^d \to \R$ are assumed to be Lipschitz continuous. \\ \\
(b)  Define the matrix
\bee \label{defM}
\mathcal{M}_\text{Lip}(\theta,\vartheta):= 
\left( \| ((\dot b_\theta \dot b_\theta^{\star})_{ij}-(\dot b_{\vartheta} \dot b_{\vartheta}^{\star})_{ij})(\cdot) \|_{\text{Lip}}
\right)_{i,j=1}^p. 
\eee
We introduce two distance measures on the parameter set $\Te$ 
\begin{align*}
\label{defd}
d_1(\theta,\vartheta)&:=\max_{j=1,...,p} \left(\E_{\te_0} \left[ \| \dot b_{\te,j}(X_0) - \dot b_{\vartheta,j}(X_0)\|^2 \right] + \| f_{\theta,\vartheta}^j \|_{\text{\rm Lip}} \right)^{\frac 12} \\[1.5 ex]
d_2(\theta,\vartheta)& := \|\mathcal{M}_\text{Lip}(\theta,\vartheta)\|_{\text{op}}
\end{align*}
and assume that 
\bee
\De_1(\Te)=\De(\Te, d_1) < \infty \qquad \text{and} \qquad \De_2(\Te)=\De(\Te, d_2) < \infty.
\eee
(c) Let $l_{\text{min}}(\te)$ denote the smallest  eigenvalue of the Fisher information matrix $I(\te)$. We assume that 
\bee \label{deffrakl}
l_{\text{min}}:= \inf_{\te \in \Te} l_{\text{min}}(\te)>0.
\eee
\newline

\noindent
The distance measures $d_1$ and $d_2$ will appear to be crucial to obtain upper bounds 
for the probabilities $\P(\mathcal{T})$ and $\P(\mathcal{T}')$ via the generic chaining method. 
We also remark that the assumption 
($\mathcal{B}$)(c) is stronger than the identifiability assumption in ($\mathcal{A}$)(c). 

\begin{rem} \label{remdistance} \rm
While the distance measures $d_1$ and $d_2$ are rather unusual, they can be upper bounded 
by more classical ones. Consider the additional assumption
\bee \label{otherbounds}
r_1:=\sup_{\te \in \Te, x\in \R^d} \max_{i,j} |\dot{b}_{\te,ij}(x)| <\infty, \quad
 r_2:=\sup_{\te \in \Te} \max_{i,j} \|\dot{b}_{\te,ij}\|_{\text{Lip}} <\infty.
\eee
We then conclude that
\bee
d_1(\theta,\vartheta) \leq \sqrt{2r_1} \max_{j=1,...,p} \left( \sum_{i=1}^d \|\dot b_{\te,ji} - 
\dot b_{\vartheta,ji}\|_{\infty} + \sum_{i=1}^d \|\dot b_{\te,ji} - 
\dot b_{\vartheta,ji}\|_{\text{\rm Lip}} \right)^{\frac 12}. 
\eee
Similarly, using the estimate $\|\mathcal{M}_\text{Lip}(\theta,\vartheta)\|_{\text{op}}\leq \|\mathcal{M}_\text{Lip}(\theta,\vartheta)\|_{\text{F}}$
where $\|\cdot\|_{\text{F}}$ denotes the Frobenius norm, we deduce the inequality
\bee
d_2(\theta,\vartheta) \leq p \times \max_{j,l=1,...,p} \mathcal{M}_\text{Lip}(\theta,\vartheta)_{jl}
\leq 2p r_1 \times  \max_{j=1,...,p} \left(  \sum_{i=1}^d \|\dot b_{\te,ji} - 
\dot b_{\vartheta,ji}\|_{\text{\rm Lip}}   \right).
\eee
In particular, it holds that 
\bee \label{deltaest}
\De_1(\Te) \leq K_{r_1,r_2} d^{1/2} \qquad \text{and} \qquad \De_2(\Te) \leq K_{r_1,r_2} d p
\eee
for a constant $K_{r_1,r_2}>0$. 
\qed
\end{rem}

\noindent
To control the probabilities of the sets $\mathcal{T}$ and $\mathcal{T}'$ 
we define the constants
\begin{align} 
\lambda_1&:=  4L\Delta_1(\Theta)\max{\left( \sqrt{\frac{\log{(2Lp)}+\log{(2/\epsilon)}}{2T}}, \left(\frac{\frac{(2C)^{1/3}}{6}\left(\log{(2Lp)+\log{(2/\epsilon)}}  \right)}{T}  \right)^{3/4} \right)}
\nonumber \\[1.5 ex]
&+ \frac{4L\gamma_{4/3}(\Theta,d_1)}{\sqrt{T}}, \label{lambda1}
\end{align}
and
\bee \label{defT1}
T_1:= \frac{2592(c_0+2)^4L^2 C}{l_{\text{min}}^2}
\left(  
\Delta_2(\Theta)
\sqrt{\log{\left( 21^{2s_0}\left( p^{2s_0} \wedge \frac{ep}{2s_0}  \right)^{2s_0} \right)}+\log{\frac{L}{\epsilon_0}}}
+\gamma_2(\Theta,d_2)
\right)^2.
\eee
where $\epsilon>0$ is a given number, $L>0$ is defined in \eqref{defL} and $C$ has been introduced in Theorem \ref{expbound}. The main result of this section is the following theorem.

\begin{theo} \label{controlT}
Assume that conditions  \eqref{bcond2}, ($\mathcal{B}$) are satisfied and fix $\ep \in (0,1)$. Then, for all $T\geq \frac{2C}{27}$ and $\lambda\geq \lambda_1$,
it holds that 
$$\P(\mathcal{T})\geq 1- \ep.$$ 
Furthermore, for any $T\geq T_1$ and $k=\sqrt{l_{\text{min}}/2}$, we obtain the bound 
$$\P(\mathcal{T}')\geq 1- \ep.$$
\end{theo}

\begin{proof}
See Section \ref{sec6}.
\end{proof}

\begin{rem} \rm
Here we present a class of drift functions $b_{\te}$ that satisfies conditions \eqref{bcond2} and  ($\mathcal{B}$)(a) and (b). In particular, this class will cover the type of functions discussed in
Example  \ref{ex1}(ii) and (iii). We consider drift functions of the form 
\[
b_{\te} (x) = \phi_1(x) + \phi_2(\te,x),
\]
where $\phi_1 \in C^1(\R^d; \R^d)$ and $\phi_2 \in C^{1,1}(\Te \times\R^d; \R^d)$. We
assume that $\phi_1$ satisfies the condition
\[
(\phi_1(x)-\phi_1(y))^{\star}(x-y) \geq 2M \| x-y \|^2 
\]
and
\[
\sup_{\te\in \Te}|(\phi_2(\te,x)-\phi_2(\te,y))^{\star}(x-y)|\leq M \| x-y \|^2.
\]
Under these assumptions we conclude that the condition \eqref{bcond2} is automatically satisfied. Furthermore, when assumption \eqref{otherbounds} is fulfilled for the function 
$\phi_2$, conditions ($\mathcal{B}$)(a) and (b) are also satisfied as noted in Remark 
\ref{remdistance}.

We note that the Ornstein-Uhlenbeck model discussed in Example  \ref{ex1}(i) does not satisfy our assumptions as the function $\dot b_\theta \dot b_\theta^{\star}$ is quadratic in $x$ and hence not Lipschitz continuous. However, in this particular case Malliavin techniques have been applied in \cite{CMP20} to obtain the necessary concentration inequalities.  
\qed
\end{rem}

\section{Error bounds for the Lasso estimator} \label{sec5}
\setcounter{equation}{0}
\renewcommand{\theequation}{\thesection.\arabic{equation}}

Now we are ready to demonstrate the error bounds for the Lasso estimator $\thetal $. 

\begin{cor} \label{maincor}
Let $\|\te_0\|_0=s_0$. Assume that conditions  \eqref{bcond2}, ($\mathcal{B}$) are satisfied and fix $\ep \in (0,1)$. If $T\geq \max\{2C/27,T_1\}$ and $\lambda\geq \lambda_1$ we obtain with 
probability at least $1-2\ep$: 
\begin{align}
\| \thetal - \theta_0\|_2^2 &\leq 
 \frac{64(\gamma+2)^2 s_0 \lambda^2}{\gamma l_{\text{min}}^2}, \\[1.5 ex]
\|  \thetal - \theta_0\|_1& \leq  \frac{32(\ga+1)(\ga+2) s_0 \la}{\ga^{3/2} l_{\text{min}}}, \\[1.5 ex]
\| \thetal\|_0 & \leq \frac{64\mathfrak{M}_\infty(\ga+1)(\ga+2)}{\ga^{3/2} l_{\text{min}}} s_0, 
\end{align}

\noindent where $\mathfrak{M}_\infty= \| \sup_{\theta,\vartheta \in \Theta } \frac{1}{T}\int_0^T \dot b_\theta^{\star}(X_t) b_\vartheta(X_t) dt\|_\infty $. In particular, if $\log(p)/T \to 0$ we deduce that 
\begin{align}
\| \thetal - \theta_0\|_2 &= O_{\P}\left(\frac{\sqrt{s_0}}{l_{\text{min}}} \left(\De_1(\Te) \sqrt{\frac{\log(p)}{T}} + \frac{\ga_{4/3}(\Te,d_1)}{\sqrt{T}} \right) \right), \\[1.5 ex]
\|  \thetal - \theta_0\|_1& = O_{\P}\left(\frac{s_0}{l_{\text{min}}} \left(\De_1(\Te) \sqrt{\frac{\log(p)}{T}} + \frac{\ga_{4/3}(\Te,d_1)}{\sqrt{T}} \right) \right).
\end{align}
\end{cor}

\begin{proof}

First two inequalities follow from the bounds \eqref{estTdis}-\eqref{l1bound} and the statement of Theorem \ref{controlT}, so we just need to show the last statement. We notice that necessary and sufficient condition for $\thetal$ to be the solution of the optimisation problem \eqref{lasso} is the existence of a vector $\vartheta$ in subdifferential of target function that satisfies

\bee
\frac{1}{T}\int_0^T \dot b_{\thetal}^{\star}(X_t)dW_t+ \frac{1}{T} \int_0^T \dot b_{\thetal}^{\star}(X_t)(b_{\thetal}(X_t)-b_{\theta_0}(X_t))dt+\lambda \vartheta=0. 
\eee

\noindent
Furthermore $\thetal^i\neq 0$ implies that $\vartheta^i=\text{sign }{\thetal^i}$ and we conclude that 

\begin{align*}
\mathfrak{M}_\infty \|  \thetal - \theta_0\|_1 \geq \| \frac{1}{T} \int_0^T \dot b_{\thetal}^{\star}(X_t)(b_{\thetal}(X_t)-b_{\theta_0}(X_t))dt \|_1 &= \| \frac{1}{T}\int_0^T \dot b_{\thetal}^{\star}(X_t)dW_t + \vartheta \|_1  \\
\geq \sum_{i: \thetal^i\neq 0} |\la \vartheta^i-|(\frac{1}{T}\int_0^T \dot b_{\thetal}^{\star}(X_t)dW_t)^i||&\geq \| \thetal \|_0 \frac{\la}{2}. 
\end{align*}
This completes the proof of Corollary \ref{maincor}.
\end{proof}

\noindent
The quantity $\ga_{4/3}(\Te,d_1)$ appearing in Corollary  \ref{maincor} is hard to compute explicitly, but it can be upper bounded using \textit{covering numbers}. To demonstrate ideas we introduce the parametric class 
\[
F:=\left\{\dot b_{\te}:~\te \in \Te \right\}
\] 
and denote by $N(F,d_1,\ep)$ the covering number of $F$ with respect to $d_1$-distance, i.e. the smallest number of balls of radius $\ep$ needed to cover the set $F$. Then one can always approximate
\[
\ga_{\al}(\Te,d_1) \leq c_{\al} \int_{0}^{\infty} \left( \log N(F,d_1,\ep) \right)^{1/\al} d\ep,
\]
for some positive constant $c_{\al}$, see e.g. \cite{Talagrand2005}. The order of covering numbers has been computed for numerous examples and these studies can be applied to control 
the entropy constant $\ga_{4/3}(\Te,d_1)$. 

Assume, for instance, that $\Te \subset \R^p$ is a bounded set, in which case $N(\Te,\|\cdot\|_2,\ep) \sim \ep^{-p}$. When the map $\te \mapsto \dot b_{\te}$ is Lipschitz with Lipschitz constant $R>0$, then we deduce the inequality 
$N(F,d_1,\ep) \leq N(\Te,\|\cdot\|_2,\ep/R)$. In this case we obtain
\bee \label{gammaex}
\ga_{4/3}(\Te,d_1) \leq c_{4/3,R}~ p^{3/4}. 
\eee
This bound may be far from optimal, but it can be useful if $p\ll d$.

On the other hand, when the dimension $d$ of the process $X$ is constant (or significantly smaller than $p$), the covering number $N(F,d_1,\ep)$ can be approximated ignoring the dependence on the parameter space $\Te$.  To demonstrate ideas we make the additional assumption that
$X_t \in Q$ for all $t\geq 0$ $\P$-almost surely and $\text{supp}(\dot b_{\te})
\subset Q$ for all $\te \in \Te$, where $Q\subset \R^d$ is a compact set, and $\dot b_{\te}(\cdot) \in C^{\be}_{Q}(\R^d; \R^d)$ for some $\be>1$ (class of $\be$-smooth functions with support $Q$).   As observed in Remark \ref{remdistance}, we have 
$d_1(\theta,\vartheta) \leq d'(\dot b_{\te}, \dot b_{\vartheta})$ where 
\[
d'(\dot b_{\te}, \dot b_{\vartheta}):=\max_{j=1,...,p} \left( \sum_{i=1}^d \|\dot b_{\te,ji} - 
\dot b_{\vartheta,ji}\|_{\infty} + \sum_{i=1}^d \|\dot b_{\te,ji} - 
\dot b_{\vartheta,ji}\|_{\text{\rm Lip}} \right)^{\frac 12}
\]
Now, we can use the estimate
\[
N(F,d_1,\ep) \leq N(C^{\be}_{Q}(\R^d; \R^d),d',\ep), 
\]
which effectively ignores the parametric structure of the set $F$. The latter covering number can be computed in a similar fashion as demonstrated  
in \cite[Section 2.7]{VW96}, and thus we may obtain an explicit upper bound for the quantity $\ga_{4/3}(\Te,d_1)$.

In the next proposition we derive a lower bound for the estimation problem.

\begin{lem}[Lower bound]
Assume that $\Te$ contains an open ball centred at $0\in \R^p$ and 
\bee
l_{max} := \sup_{\te,\vartheta \in \Te}  \left\|\E_{\te}[\dot b_{\vartheta}(X_0) \dot b_{\vartheta}(X_0)^{\star}] \right\|_{\text{\rm op}}<\infty.
\eee
Then it holds that 
\bee
\inf_{\widehat{\te}} \sup_{\te \in \Te:~\|\te\|_0=s_0} \E_{\te}[\|\widehat{\te}-\te\|_2] \geq c' 
\sqrt{\frac{s_0}{l_{max}T}}
\eee
for some $c'>0$. 
\end{lem}

\begin{proof}
Recall that $\P_{\theta}^T$ denotes  the law of the process \eqref{DP} with parameter $\theta$ restricted to $\f_{T}$. We will now apply \cite[Theorem 2.2]{T09}. For this purpose,  assume 
that $\te_1,\te_2 \in \{0,\ep\}^p$ are two parameters with $\|\te_1\|_0 = \|\te_2\|_0=s_0$ such that 
\[
\|\te_1-\te_2\|_2 \geq c s_0 \ep
\]
for some $c\in (0,1)$. We approximate the Kullback-Leibler divergence $K(\P_{\theta_1}^T|
\P_{\theta_2}^T)$ as follows:
\begin{align}
K(\P_{\theta_1}^T|
\P_{\theta_2}^T) &\leq \frac{T}{2} \E_{\te_1}\left[\left|\|b_{\te_1}(X_0)\|_2^2 - \|b_{\te_2}(X_0)\|_2^2
\right| \right] \\
&\leq \frac{l_{max} T \|\te_1-\te_2\|^2}{2} \leq l_{max}Ts_0\ep^2. 
\end{align}
Choosing $\ep^2=O((l_{max}Ts_0)^{-1})$ and applying \cite[Theorem 2.2]{T09}, we obtain the desired statement. 
\end{proof}

\noindent
We emphasise the dependence of the lower bound on $l_{max}$, which can grow in the dimension of the model $d$. This dependence is rather intuitive as the higher dimensionality of the model should usually improve the convergence rate. From this perspective it is interesting to compare the lower bound with convergence rates obtained in Corollary \ref{maincor}. For this purpose assume that 
\bee \label{condonlim}
c_1 d\leq l_{\min} \leq l_{\max} \leq c_2 d,
\eee   
which holds, for instance, for the independent case (see example below).  In this setting, when we use the estimates \eqref{deltaest} and   \eqref{gammaex}, we deduce that 
\bee 
\|\thetal - \theta_0\|_2 = O_{\P}\left(\sqrt{\frac{s_0\log(p)}{dT}} + 
p^{3/4} \sqrt{\frac{s_0}{d^2 T}}  \right).
\eee
If $p^{3/4} \ll d^{1/2}$ the first term is dominating and it matches the lower bound up to $\log(p)$
factor. 

\begin{ex} \rm
While the rates of convergence for the estimator $\thetal$ may be suboptimal in some cases, we give simple examples where the rates are nearly optimal. \\
(i) (Linear case) As mentioned in Remark \ref{rem3.3} in linear drift models of Example \ref{ex1}(ii) uniform bounds are not required and hence the quantities $\Delta_{j}(\Te)$ and $\ga_{\al}(\Te, d_{j})$.,
$j=1,2$, can be replaced by $1$. In this case, under condition \eqref{condonlim},  the bounds simplify to  
\[
\| \thetal - \theta_0\|_2 = O_{\P} \left( \sqrt{s_0 \frac{\log(p)}{dT}} \right)
\qquad \text{and} \qquad
\| \thetal - \theta_0\|_1 = O_{\P} \left( s_0\sqrt{ \frac{\log(p)}{dT}} \right).
\] 
These are rather classical bounds, which match the lower bounds up to 
$\log(p)$ factor. \\
(ii) (Independent case) Here we consider the special case of independent particles. More specifically, we assume that the drift function has the form $b_{\theta}(x)= 
(g_{\theta}(x_1),\ldots, g_{\theta}(x_d))^{\star}$, where  $g_{\theta}:\R \to \R$ is a given function. In the case there is no interaction between diffusion particles and hence they are independent and identically distributed. If we assume that the individual information matrix 
\[
I_1(\theta):= \E_{\te}[\dot g_{\te}(X_0) \dot g_{\te}(X_0)^{\star}]
\] 
is invertible,  we deduce that $I(\te)=d I_1(\theta)$ and condition \eqref{condonlim} is satisfied.
 \qed  
\end{ex}

\begin{rem}[Adaptive Lasso] \rm
For the purpose of model selection it is a common practice to consider the \textit{adaptive Lasso}. If we denote by $\widehat{\te}_{\text{o}}$ an initial Lasso estimator of $\te_0$, the adaptive Lasso is defined as
\[
\widehat{\te}_{\text{adap}}= \argmin_{\theta \in \R^p}\left( \mathcal{L}_T(\theta) + \lambda \sum_{j=1}^{p} \frac{\te_j}{|\widehat{\te}_{\text{o}}|^{\al}}\right)
\] 
for some $\al>0$. The support recovery property for adaptive Lasso estimator has been shown in \cite{GM19} for the setting of Ornstein-Uhlenbeck models 
with \textit{fixed} $d$ and $p$; we also refer to \cite{PI14} for the study of \textit{linear} drift models. However, showing this property for general drift functions and 
large $d$ and $p$ is not a trivial task, and we leave it for future research.    \qed
\end{rem}

\section{Numerical experiments} \label{secNE}

Our Lasso estimator is based on continuous observations of the underlying process, which have to be discretised for numerical simulations. We will use time of observation $T=20$ with 100 discretisation points over each unit interval, and $d=10$. Further refinement of the grid does not lead to a significant improvement of simulation, and consequently such approximation is sufficient for our purposes. 

For $(x,k)\in \mathbb{R}^2$ we introduce the  function

\bee
h(x,k)=\begin{cases}
               x^2 \sin{\frac{k}{x}}, \text{ if }x\neq 0\\
               0, \text{ if }x=0
            \end{cases}
\eee

\noindent and for a positive-definite $d\times d$-matrix $A$ we define a function $b_A:\R^d\to\R^d$ as

\bee
b_A^{(i)}(x)=\sum_{j=1}^d h(x_j,A_{ij}), \text{ where }x\in \R^d \text{ and }i\in\{1,...,d\}.  
\eee

\noindent
As an example, we are going to consider a diffusion process \eqref{DP} with drift function $b_{A_0}$ with some sparse parameter $\theta_0=\text{vect}(A_0)$. We set sparsity of $\theta_0$ at 35\%.   

The selection of a tuning parameter $\lambda$ is usually performed by cross-validation algorithm. We use 80\% of our observations as training set and last 20\% as validation set: 

\bee
\la_0=\argmin_{\la >0} \mathcal{L}_{[0.8T,T]}\hat{A}_L(\la), 
\eee

\noindent where 

\bee
\hat{A}_L(\la)=\argmin_{A\in R^{d \times d}} ( \mathcal{L}_{[0,0.8T]}(A)+\la \|A\|_1). 
\eee

\noindent
Figure 1 below demonstrates an example of the parameter matrix $A_0\in \R^{d\times d}$ and the corresponding maximum likelihood and Lasso estimators. We use a colour code to highlight the sparsity of given matrices. We can see that performance 
 of Lasso estimator, especially in terms of support recovery, is much superior to that one of MLE, even for a relatively small dimension of true parameter. 

\begin{figure}[!h]
\centering
\subfloat[True parameter]{
  \includegraphics[width=65mm]{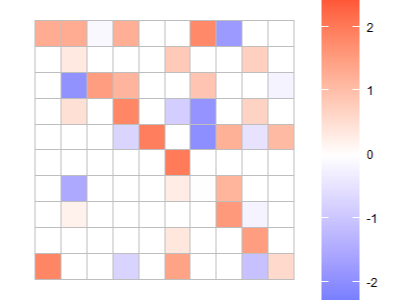}
}
\hspace{6pt}
\subfloat[MLE]{
  \includegraphics[width=65mm]{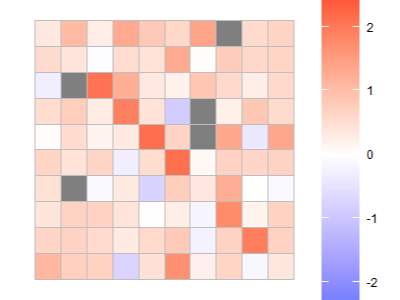}
}
\subfloat[Lasso]{
  \includegraphics[width=65mm]{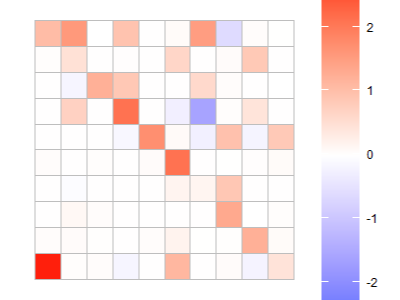}
}

\caption{Comparison of the sparsed true parameter matrix with maximum likelihood and Lasso estimators.}
\label{Figure1}
\end{figure}

\section{Proof of Theorem \ref{controlT}} \label{sec6}
\setcounter{equation}{0}
\renewcommand{\theequation}{\thesection.\arabic{equation}}

We divide the proof of Theorem \ref{controlT} into several steps.

\subsection{Treatment of the set $\mathcal{T}$} \label{sec6.1}

We consider  univariate martingale
terms
\bee \label{martin}
M_T^j (\te) := \frac{1}{\sqrt{T}} \int_0^T \big( \dot b_{\te,j}(X_t))^{\star} dW_t, \qquad j=1,\ldots,p.
\eee 
The first result determines bounds for exponential moments of $M_T^j (\te)$.
\begin{prop} \label{Propexp}
Assume that conditions  \eqref{bcond2}, ($\mathcal{B}$) are satisfied. 
Then it holds that
\bee \label{expboundpq}
\E_{\te_0}\left[ \exp\left( \mu \left( M_T^j (\te) -M_T^j (\vartheta) \right) \right) \right]
\leq \exp\left(\mu^2 d_1^2(\theta,\vartheta)
+\frac{2C  }{T } \mu^4 d_1^4(\theta,\vartheta) \right)
\eee
for any $\mu \in \R$.
\end{prop}

\begin{proof}
We first recall that due to the property of an exponential martingale we have that 
\[
\E_{\te_0}\left[a(M_T^j (\te) -M_T^j (\vartheta)) - \frac{a^2}{2} \lan M^j (\te) -M^j (\vartheta) \ran_T \right] =1
\]
for any $a\in \R$. Next, using Cauchy-Schwarz inequality, we deduce that 
\begin{align}
\E_{\te_0}\left[ \exp\left( \mu \left( M_T^j (\te) -M_T^j (\vartheta) \right) \right) \right] 
&= \E_{\te_0}\left[ \exp\left( \mu \left( M_T^j (\te) - M_T^j (\vartheta)\right) 
- \mu^2\lan M^j (\te) -M^j (\vartheta) \ran_T \right) \right. \\
& \times \left.
\exp\left( \mu^2 \lan M^j (\te) -M^j (\vartheta) \ran_T \right)
\right] \\
&\leq \E_{\te_0}\left[ \exp\left( 2\mu^2 \lan M^j (\te) -M^j (\vartheta) \ran_T  \right)  \right]^{1/2}.
\end{align} 
Hence, we conclude from Theorem \ref{expbound} that 
\[
\E_{\te_0}\left[ \exp\left( \mu \left( M_T^j (\te) -M_T^j (\vartheta) \right) \right) \right] \leq \exp\left(\mu^2 d_1^2(\theta,\vartheta)
+\frac{2C  }{T } \mu^4 d_1^4(\theta,\vartheta) \right),
\]
which completes the proof. 
\end{proof}

\noindent
An immediate consequence of Proposition \ref{Propexp} is the following concentration inequality.

\begin{cor} \label{concM}
Assume that conditions  \eqref{bcond2}, ($\mathcal{B}$) are satisfied. 
Then we obtain the inequality
\bee \label{nonunbound}
\P_{\te_0} \left( | M_T^j (\te) -M_T^j (\vartheta)| > u d_1(\theta,\vartheta) \right) \leq 
  \begin{cases}
      2 \exp{\left( - \frac{u^2}{32} \right)}, \quad  0< u \leq 8 \sqrt{\frac{T}{2C}}\\
      2 \exp{\left( - \frac{3}{8} \frac{1}{(2C)^{1/3}} T^{1/3} u^{4/3} \right)}, \quad  u > 8 \sqrt{\frac{T}{2C}}
    \end{cases}. 
\eee
\end{cor}

\begin{proof}
Due to Chebyshev's inequality and \eqref{expboundpq} we can bound 
\begin{equation}\label{ChebyshevIneq}
\P_{\te_0} \left( | M_T^j (\te) -M_T^j (\vartheta)| > u d_1(\theta,\vartheta) \right) \leq \exp \left( \mu^2 d_1^2(\theta,\vartheta) + \frac{2C}{T} \mu^4 d_1^4(\theta,\vartheta) - u \mu d_1(\theta,\vartheta) \right) 
\end{equation}
for all $\mu>0$. In particular, for $0<\mu\leq\sqrt{\frac{T}{2C}}\frac{1}{d_1(\theta,\vartheta)}$ the minimum of right side of \eqref{ChebyshevIneq} is obtained for $\mu=\frac{u}{4d_1(\theta,\vartheta)}$ and, consequently, 
\begin{equation}\label{IneqSmallu}
\P_{\te_0} \left( | M_T^j (\te) -M_T^j (\vartheta)| > u d_1(\theta,\vartheta) \right) \leq \exp \left( -\frac{u^2}{8} \right), \quad \text{if} \quad u\leq 4 \sqrt{\frac{T}{2C}}. 
\end{equation}
Similarly, for $\mu > \sqrt{\frac{T}{2C}}\frac{1}{d_1(\theta,\vartheta)}$ the minimum of right side of \eqref{ChebyshevIneq} is obtained for $\mu=\left( \frac{T}{16C}\frac{u}{d_1^3(\theta,\vartheta)} \right)^{1/3}$ and 
 \begin{equation}\label{IneqLargeu}
\P_{\te_0} \left( | M_T^j (\te) -M_T^j (\vartheta)| > u d_1(\theta,\vartheta) \right) \leq \exp{\left( - \frac{3}{8} \frac{1}{(2C)^{1/3}} T^{1/3} u^{4/3} \right)}, \quad \text{if} \quad u\geq 8 \sqrt{\frac{T}{2C}}. 
\end{equation} 
 Finally, rescaling \eqref{IneqSmallu} for $u\in \left( 4 \sqrt{\frac{T}{2C}}, 8 \sqrt{\frac{T}{2C}} \right)$ and combining it with \eqref{IneqLargeu} finishes the proof.
 \end{proof}

\noindent
We now use the generic chaining method to obtain a uniform concentration inequality for the martingale part via Corollary \ref{concM}.

\begin{prop} \label{realfinallem}
Assume that conditions  \eqref{bcond2}, ($\mathcal{B}$) are satisfied. 
For all $T\geq \frac{2C}{27}$ we deduce the inequality  
\begin{gather}
\P_{\te_0}\bigg(  \sup_{\theta\in\Theta}  | M_T^j (\te) -M_T^j (\te_0)| \geq L \left( \gamma_{4/3}(\Theta,d_1)+u \Delta_1(\Theta) \right) \bigg)\\
\leq L \exp{\left( -2 u^2  \right)}+L\exp{\left( -\frac{6}{(2C)^{1/3}} T^{1/3}u^{4/3}  \right)}, 
\end{gather}
\noindent where $L>0$ is a constant independent of $p$, $d$ and $T$. 
\end{prop}

\begin{proof}
We apply the generic chaining technique to obtain uniform bounds. Consider an admissible sequence $(\mathcal{A}_n)$ of partitions such that 
\bee
\sup_{\theta \in \Theta} \sum_{n\geq 0} 2^{3n/4} \Delta(A_n(\theta),d_1)\leq 2\gamma_{4/3}(\Theta,d_1). 
\eee
For each $n\geq 0$ consider a set $\Theta_n$ with $\text{card }{\Theta_n}\leq 2^{2^n}$ such that every element $A$ of $\mathcal{A}_n$ meets $\Theta_n$ and define $\Psi_n=\cup_{i\leq n} \Theta_n$. For $u>0$ consider the random event   
\begin{gather*}
\Omega(u)= \big\{ \forall n \geq 1, \quad \forall \theta, \vartheta \in \Psi_n , \quad |M_T^j (\te) -M_T^j (\vartheta)|\leq 8(2^{3n/4}+u)d_1(\theta,\vartheta) \big\}. 
\end{gather*}
Then, using the union bound, we can estimate 
\begin{gather}
\P_{\te_0} (\Omega^c(u))\leq \sum_{n\geq 1} (\text{card } \Psi_n)^2 \P \left( |M_T^j (\te) -M_T^j (\vartheta)| > 8(2^{3n/4}+u)d_1(\theta,\vartheta) \right) \\
\leq 8 \sum_{n\geq 1} \exp{\left( 2^n \log{2}-2^{3n/2+1} \right)}\exp{\left( -2u^2 \right)} \\
+ 8 \sum_{n\geq 1} \exp{\left( 2^n \log{2}-\frac{6}{(2C)^{1/3}} T^{1/3} 2^n \right)} \exp{\left( -\frac{6}{(2C)^{1/3}} T^{1/3}u^{4/3}  \right)}\\
\leq  L_0\left( \exp{\left( -2 u^2  \right)}+\exp{\left( -\frac{6}{(2C)^{1/3}} T^{1/3}u^{4/3}  \right)} \right)
, \label{tes}
\end{gather}
where
\bee \label{defL}
L_0:= 8 \sum_{n\geq 1} \exp{\left( 2^n \log{2}-2^{3n/2+1} \right)}
\eee
and the last inequality holds if $T\geq \frac{2C}{27}$. Now, we will follow the approach similar to proof of  \cite[Theorem 2.2.27]{Talagrand2014}, to show that when $\Omega(u)$ occurs, we have 
\bee
\sup_{\theta\in\Theta}  | M_T^j (\te) -M_T^j (\te_0)| \leq L \left( \gamma_{4/3}(\Theta,d_1)+u \Delta_1(\Theta) \right).
\eee

\noindent
Consider $\vartheta\in\Theta$ and define by induction over $q\geq 0$ integers $n(\vartheta,q)$ as $n(\vartheta,0)=0$ and for $q\geq 1$

\bee
n(\vartheta,q)=\inf\{n \text{ }| \text{ } n\geq n(\vartheta,q-1) \text{ and } d_1(\vartheta, \Psi_{n})\leq \frac{1}{2} d_1(\vartheta, \Psi_{n(\vartheta,q-1)})\}. 
\eee

\noindent
We then consider $\pi_q(\vartheta)\in \Psi_{n(\vartheta,q)}$ with $d_1(\vartheta,\pi_q(\vartheta))=d_1(\vartheta,\Psi_{n(\vartheta,q)})$. Thus by induction it holds that 

\bee\label{pitheta}
d(\vartheta, \pi_q(\vartheta))\leq 2^{-q}\Delta_1(\Theta). 
\eee

\noindent Also when $\Omega(u)$ occurs we obtain

\bee
|M_T^j (\pi_q(\vartheta)) -M_T^j (\pi_{q-1}(\vartheta))|\leq 8(2^{3n(\vartheta,q)/4}+u)d_1(\pi_q(\vartheta),\pi_{q-1}(\vartheta)).
\eee

\noindent and 

\begin{align*}
|M_T^j (\theta) -M_T^j (\theta_0)|& \leq \sum_{q\geq 1} |M_T^j (\pi_q(\vartheta)) -M_T^j (\pi_{q-1}(\vartheta))| \\
& \leq \sum_{q\geq 1} 8(2^{3n(\vartheta,q)/4}+u)d_1(\pi_q(\vartheta),\pi_{q-1}(\vartheta))\\
& \leq \sum_{q\geq 1} 8(2^{3n(\vartheta,q)/4}+u)d_1(\vartheta, \pi_q(\vartheta))\\
&+ \sum_{q\geq 1} 8(2^{3n(\vartheta,q)/4}+u)d_1(\vartheta, \pi_{q-1}(\vartheta))
\end{align*}

\noindent We now control each summation separately. First  
\begin{align*}
\sum_{q\geq 1} 2^{3n(\vartheta,q)/4+3} d_1(\vartheta, \pi_q(\vartheta)) &\leq \sum_{q\geq 1} 2^{3n(\vartheta,q)/4+3} d_1(\vartheta, \Theta_{n(\vartheta,q)}) \leq 16 \gamma_{4/3}{(\Theta,d_1)}.
\end{align*}

\noindent  Now by the definition 

\bee
d_1(\vartheta,\pi_{q-1}(\vartheta))=d_1(\vartheta, \Psi_{n(\vartheta,q-1)})\leq 2d_1(\vartheta, \Psi_{n(\vartheta,q)-1})
\eee

\noindent so that

\bee
\sum_{q\geq 1} 2^{3n(\vartheta,q)/4+3}d_1(\vartheta, \pi_{q-1}(\vartheta)) \leq  \sum_{q\geq 1} 2^{3n(\vartheta,q)/4+4}d_1(\vartheta, \Psi_{n(\vartheta,q)-1})  \leq 2^{23/4} \gamma_{4/3}{(\Theta,d_1)}. 
\eee

\noindent Finally, we can notice that \eqref{pitheta} implies 

\bee
\sum_{q\geq 1} 8 (d_1(\vartheta, \pi_q(\vartheta))+d_1(\vartheta, \pi_{q-1}(\vartheta))) \leq 24 \Delta(\Theta),
\eee

\noindent which finishes the proof for $L=\max(16+2^{23/4},L_0)=16+2^{23/4}$. 

\end{proof}

\noindent Now, we apply the statement of Proposition \ref{realfinallem} to control $\P(\mathcal{T})$. We recall the definition of $\la_1$ at  \eqref{lambda1}.
Due to previous proposition, we deduce via the union bound that  
\begin{gather}
\P_{\te_0}\left( \left \| \sup_{\te \in \Te} \frac{1}{T} \int_0^T \big( \dot b_{\te}(X_t))^{\star} dW_t  \right \|_{\infty} \geq \frac{\lambda}{2} \right)
\leq \sum_{j=1}^p \P_{\te_0}\left( \left|\sup_{\te \in \Te} \frac{1}{\sqrt{T}} \int_0^T \big( \dot b_{\te,j}(X_t))^{\star} dW_t  \right | \geq \frac{\lambda}{2} \sqrt{T} \right) \\
\leq \sum_{j=1}^p \left( \P_{\te_0}\left(  \sup_{\theta\in\Theta}  |M_T^j (\te) -M_T^j (\te_0)|)\geq \frac{\lambda}{4}\sqrt{T}\right) 
+  \P_{\te_0}\left(  |M_T^j (\te_0)|\geq \frac{\lambda}{4}\sqrt{T}\right) \right) \\
\leq 2 p L \exp{\left( -2\left( \frac{\lambda\sqrt{T}}{4L\Delta_1(\Theta)} -\frac{\gamma_{4/3}(\Theta,d_1)}{\Delta_1(\Theta)} \right)^2\right) } \\
+ 2 p L \exp{\left( -\frac{6}{(2C)^{1/3}}T^{1/3} \left( \frac{\lambda\sqrt{T}}{4L\Delta_1(\Theta)} -\frac{\gamma_{4/3}(\Theta,d_1)}{\Delta_1(\Theta)} \right)^{4/3}\right) }  \leq \epsilon, 
\end{gather}

\noindent where the last inequality holds since $\lambda \geq \lambda_1$. This completes the proof of the first part of Theorem \ref{controlT}.

\subsection{Treatment of the set $\mathcal{T}'$} \label{sec6.2}

We start by defining the random 
function ($u\in \R^p$ with $\|u\|_2=1$)
\bee \label{Hdef}
H_{u}(\te):= \frac{u^{\star}}{T} \int_0^T \left(\dot b_{\te}(X_t)\dot b_{\te}(X_t)^{\star} - \E_{\te_0}[\dot
b_{\te}(X_t)  \dot b_{\te}(X_t)^{\star} ]\right)dt \times u.
\eee
We first show a uniform probability bound for the process $H$.

\begin{lem}\label{LemmaTauPrime}Assume that conditions  \eqref{bcond2}, ($\mathcal{B}$) are satisfied. Then we obtain the bound
\bee
\sup_{u \in \Theta} \P_{\te_0}\bigg( \sup_{\xi\in\Theta}\frac{|H_{u}(\te_0)-H_{u}(\xi)|}{\|  u \|_2^2} \geq y \bigg) \leq L \exp \left( -\frac{1}{\Delta^2_2(\Theta)} \left( \frac{y}{L}\sqrt{\frac{T}{2C}}-\gamma_2(\Theta,d_2) \right)^2 \right)
\eee
where $L$ has been introduced at \eqref{defL}. 
\end{lem}

\begin{proof}

For $u\in \Theta$, $\|u\|_2=1$ lets consider functional $F_{u,\theta,\xi}:C([0,T],\R^d)\to\R$ such that 
\bee
F_{u,\theta,\xi}(X[0,T])=|H_{u}(\te)-H_{u}(\xi)|
\eee
Then by definition the following estimate holds 
\bee\label{LambdaDist}
\|F_{u,\theta,\xi}\|_{\text{Lip}} \leq \frac{d_2(\theta,\xi)}{T} .
\eee
Then once again according to Theorem \ref{expbound} we deduce that 
\bee
\P_{\te_0} \left( \frac{|H_{u}(\theta)-H_{u}(\xi)|}{\|  u \|_2^2} \geq y d_2(\theta,\xi) \right) \leq 2 \exp\left( -\frac{T}{4C} y^2 \right).
\eee
Applying Theorem 2.2.27 in \cite{Talagrand2014} we obtain for all $u\in\Theta$:
\bee
\P_{\te_0}\bigg( \sup_{\xi\in\Theta}\frac{|H_{u}(\te_0)-H_{u}(\xi)|}{\|  u \|_2^2} \geq \sqrt{\frac{2C}{T}}L \left(\gamma_2(\Theta, \Lambda)+ y \Delta(\Theta,\Lambda) \right)
\bigg)\leq L \exp{(-y^2)}
\eee
This finishes the proof. 
\end{proof}

\noindent
Now we proceed with estimation of the probability $\P(\mathcal{T}')$.
We observe that 
\begin{gather*}
 \frac{\|b_{\theta} - b_{\vartheta} \|_{T}^2}{\|  \theta - \vartheta\|_2^2}
 = \frac{\E [\|b_{\theta} - b_{\vartheta} \|_{T}^2]}{\|  \theta - \vartheta\|_2^2} - \frac{\E [\|b_{\theta} - 
 b_{\vartheta} \|_{T}^2]-\|b_{\theta} - b_{\vartheta} \|_{T}^2}{\|  \theta - \vartheta\|_2^2} \geq \\
l_{\text{min}} -  \frac{\E [\|b_{\theta} - b_{\vartheta} \|_{T}^2]-\|b_{\theta} - b_{\vartheta} \|_{T}^2}{\|  \theta - \vartheta\|_2^2}
\end{gather*}
where the constant $l_{\text{min}}$ has been defined in  \eqref{deffrakl}.
Consequently, we deduce the inequality
\begin{gather*}
\P_{\te_0}\bigg( \inf_{\substack{\theta\neq \vartheta \in \Theta: \\ \theta-\vartheta \in \mathcal{C}(s,3+4/\gamma)}}  \frac{[\|b_{\theta} - b_{\vartheta} \|_{T}^2}{\|  \theta - \vartheta\|_2^2}\geq \frac{l_{\text{min}}}{2}\bigg) \\
\geq \P\bigg( \sup_{\substack{\theta\neq \vartheta \in \Theta: \\ \theta-\vartheta \in \mathcal{C}(s,3+4/\gamma)}}  \frac{\E [\|b_{\theta} - b_{\vartheta} \|_{T}^2]-\|b_{\theta} - b_{\vartheta} \|_{T}^2}{\|  \theta - \vartheta\|_2^2}\leq \frac{l_{\text{min}}}{2}\bigg).
\end{gather*}
Applying the mean value theorem we can rewrite for some $\xi\in\Theta$
\bee
\E_{\te_0} [\|b_{\theta} - b_{\vartheta} \|_{T}^2]-\|b_{\theta} - b_{\vartheta} \|_{T}^2= -H_{\theta-\vartheta}(\xi).
\eee
To summarize, we obtain that
\begin{gather*}
\P_{\te_0}\bigg( \inf_{\substack{\theta\neq \vartheta \in \Theta: \\ \theta-\vartheta \in \mathcal{C}(s,3+4/\gamma)}}  \frac{[\|b_{\theta} - b_{\vartheta} \|_{T}^2}{\|  \theta - \vartheta\|_2^2}\geq \frac{l_{\text{min}}}{2}\bigg) \geq 
1 - \P_{\te_0}\bigg(\sup_{u \in \mathcal{C}(s,3+4/\gamma)}  \sup_{\xi\in\Theta}\left(-\frac{H_{u}(\xi)}{\|  u \|_2^2} \right) \geq \frac{l_{\text{min}}}{2}\bigg)\geq \\
1- \P_{\te_0}\bigg( \sup_{u \in \mathcal{C}(s,3+4/\gamma)} \sup_{\xi\in\Theta}\frac{|H_{u}(\te_0)-H_{u}(\xi)|}{\|  u \|_2^2} \geq \frac{l_{\text{min}}}{4}\bigg)
-\P_{\te_0}\bigg( \sup_{u \in \mathcal{C}(s,3+4/\gamma)}  \frac{|H_{u}(\te_0)|}{\|  u \|_2^2} \geq \frac{l_{\text{min}}}{4}\bigg).
\end{gather*}
The second probability is easy to analyse, so we will concentrate on bounding the first one. Similarly to Lemmas F.1 and F.3 from the supplementary material of \cite{BM15} one can show that 
\begin{gather*}
\P_{\te_0}\bigg( \sup_{u \in \mathcal{C}(s,3+4/\gamma)} \sup_{\xi\in\Theta}\frac{|H_{u}(\te_0)-H_{u}(\xi)|}{\|  u \|_2^2} \geq \frac{l_{\text{min}}}{4}\bigg)\leq \\
\P_{\te_0}\bigg( 3(c_0+2)^2\sup_{u \in \mathcal{K}(2s)} \sup_{\xi\in\Theta}\frac{|H_{u}(\te_0)-H_{u}(\xi)|}{\|  u \|_2^2} \geq \frac{l_{\text{min}}}{4}\bigg)\leq \\
21^{2s} \left( p^{2s} \wedge \left( \frac{ep}{2s} \right)^{2s} \right) \sup_{u \in \Theta} \P_{\te_0}\bigg( \sup_{\xi\in\Theta}\frac{|H_{u}(\te_0)-H_{u}(\xi)|}{\|  u \|_2^2} \geq \frac{l_{\text{min}}}{36(c_0+2)^2}\bigg)
\end{gather*}
Consequently, the second part of Theorem  \ref{controlT} follows from  Lemma \ref{LemmaTauPrime}. This completes the proof for $L=16+2^{23/4}$.

\begin{acks}[Acknowledgments]
The authors would like to thank Claudia Strauch, Richard Nickl, Yuri Kutoyants and Jon Wellner 
for useful comments.
\end{acks}
\begin{funding}
The authors gratefully acknowledge financial support of ERC Consolidator Grant 815703
``STAMFORD: Statistical Methods for High Dimensional Diffusions''. 
\end{funding}




\end{document}